\documentclass[reqno]{amsart}
\usepackage{amsmath}

\usepackage{graphicx}
\usepackage{amscd}
\usepackage{amsmath, amssymb}
\usepackage{graphics}
\usepackage{graphicx}
\usepackage{epsfig}
\usepackage{xcolor}
\usepackage{soul}
\usepackage{tikz}

\newtheorem{theorem}{Theorem}

\newtheorem{definition}[theorem]{Definition}

\newtheorem{proposition}[theorem]{Proposition}

\newtheorem{question}[theorem]{Question}

\begin{document}

\title{Twin groups representations}

\author{Mohamad N. Nasser}

\address{Mohamad N. Nasser\\
         Department of Mathematics and Computer Science\\
         Beirut Arab University\\
         P.O. Box 11-5020, Beirut, Lebanon}
         
\email{m.nasser@bau.edu.lb}

\begin{abstract}
We construct two representations of the twin group $T_n, n\geq 2$, namely $\eta_1: T_n \rightarrow \mathrm{Aut}(\mathbb{F}_n)$ and $\eta_2: T_n \rightarrow \mathrm{GL}_n(\mathbb{Z}[t^{\pm 1}])$, where $\mathbb{F}_n$ is a free group with $n$ generators and $t$ is indeterminate. We then analyze some characteristics of these two representations, such as irreducibility and faithfulness. Moreover, we prove that both representations can be extended to the virtual twin group $VT_n$ in the $2$-local extension way, for $n\geq 2$, and we find their $2$-local extensions. On the other hand, we obtain a different result for the welded twin group $WT_n$. More deeply, we show that $\eta_1$ cannot be extended to $WT_n$ in the $2$-local extension way, for $n\geq 3$, while $\eta_2$ can be extended to $WT_n$ in the $2$-local extension way, for $n\geq 2$, and we find its $2$-local extensions.
\end{abstract}

\maketitle

\renewcommand{\thefootnote}{}
\footnote{\textit{Keywords and phrases.} Twin Groups, Braid Groups, Group Representations, Irreducibility.}
\footnote{\textit{Mathematics Subject Classification.} Primary: 20F36.}

\section{Introduction} 

The twin group $T_n$, where $n\geq 2$, is a right-angled Coxeter group with $n-1$ generators $s_1,s_2,\ldots, s_{n-1}$ and the following defining relations.
\begin{align*}
&\hspace{0.27cm} s_i^2=1,\hspace{0.75cm} i=1,2,\ldots,n-1,\\
&s_is_j=s_js_i, \hspace{0.3cm} |i-j|\geq 2.
\end{align*}
This group first appeared in the work of G. Shabat and V. Voevodsky in \cite{Shabat1990}, and then investigated by M. Khovanov, who provided a geometric interpretation similar to that of the Artin braid group (see \cite{Khovanov1996, Khovanov1997}). Then, it appeared under the name of flat braids in \cite{Merkov1999} and planar braids in \cite{Mosto2020}. The group $T_n$ and its group extensions are studied for their properties related to the braid groups, including their algorithms for checking equivalence under Markov moves, as well as their applications in the study of the fundamental group of configuration spaces (see \cite{Bellingeri2024}).

\vspace*{0.1cm}

There exist two well-known group extensions of the twin group $T_n$, namely the virtual twin group $VT_n$ and the welded twin group $WT_n$. These groups are defined in \cite{Bardakov2019} in close analogy to the virtual and welded braid groups. Each of them is generated by two distinct families of generators: the classical generators $s_1, s_2, \ldots, s_{n-1}$ corresponding to those of $T_n$, and an additional family of generators $\rho_1, \rho_2, \ldots, \rho_{n-1}$ that encode the virtual (or welded) interactions. The defining relations between these generators, which distinguish $VT_n$ from $WT_n$, are presented in Section~2. For a more detailed discussion of the structure and properties of the groups $T_n$, $VT_n$, and $WT_n$, we refer the reader to \cite{Bardakov2019, Naik2020, Bellingeri2024}.

\vspace*{0.1cm}

Representation theory is a branch of Mathematics that studies algebraic structures by representing their elements as automorphisms of a free group or as matrices. Constructing representations of any algebraic structure leads to solving many problems in other mathematical branches. For instance, constructing a faithful representation of a group can lead to solving the word problem of this group. To the best of our knowledge, the existence of faithful representations for the groups $VT_n$ and $WT_n$ remains an open problem, except in some particular or low-dimensional cases (see, for example, \cite{Bardakov2019}). On the other hand, the study of the irreducibility of a group representation is of fundamental importance, as irreducible representations often reveal deep structural properties of the underlying group. Furthermore, such representations play a central role in various mathematical and physical theories, including Quantum Mechanics (see \cite{J.Yang}).

\vspace*{0.1cm} 

One of the most famous types of matrix representations in this field is called the $k$-local representation. A representation of a group $G$ with a finite number of generators $a_1,a_2,\ldots, a_{n-1}$ into $\mathrm{GL}_m(\mathbb{Z}[t^{\pm 1}])$, where $t$ is an indeterminate, is said to be $k$-local if the image of the generator $a_i, \ 1\leq i \leq n-1,$ has the form
$$\left( \begin{array}{c|@{}c|c@{}}
   \begin{matrix}
     I_{i-1} 
   \end{matrix} 
      & 0 & 0 \\
      \hline
    0 &\hspace{0.2cm} \begin{matrix}
   		M_i
   		\end{matrix}  & 0  \\
\hline
0 & 0 & I_{n-i-1}
\end{array} \right),$$ 
where $M_i \in \mathrm{M}_k(\mathbb{Z}[t^{\pm 1}])$ with $k=m-n+2$ and $I_r$ is the $r\times r$ identity matrix. It is important to classify and study the characteristics of $k$-local representations of many algebraic structures with different degrees $k$, especially the braid group and their group and monoid extensions (see, for instance, \cite{Mikha2013,chreif2024,Mayasi2025,M.N.12,M.N.13}).
 
\vspace*{0.1cm} 

The primary objective of this article is to construct representations of the twin group $T_n$ in two different ways: as automorphisms of a free group and as matrix representations. In addition, we study several fundamental properties of these representations, including their irreducibility and faithfulness. Furthermore, we investigate the possibility of extending the constructed representations to the virtual twin group $VT_n$ and the welded twin group $WT_n$.

\vspace*{0.1cm} 

The paper is organized as follows. Section~2 provides the necessary definitions and properties used in this work. In Section~3, we construct a representation $\eta_1:T_n \to \mathrm{Aut}(\mathbb{F}_n)$, where $\mathbb{F}_n$ is the free group generated by $x_1, x_2, \ldots, x_n$ (Proposition~\ref{prop1}). From this representation, we derive a corresponding matrix representation using the Magnus representation together with Fox derivatives (Propositions~\ref{proppp} and \ref{parop}), and we investigate its properties. In particular, we show that the representation is reducible, while its $(n-1)$-dimensional composition factor over $\mathbb{C}$, under the specialization $t \in \mathbb{C}^*$, is irreducible if and only if $t \neq \dfrac{2n-2}{n-2}$ and $t \neq 2$ (Theorems~\ref{reddd} and \ref{reddd1}). Furthermore, we establish that the representation is faithful for $n=2$ and $n=3$ (Theorem~\ref{fa1}). We then classify all possible extensions of $\eta_1$ to the virtual and welded twin groups in the case $n=2$ (Theorem~\ref{n2}), and for $n\geq 3$, we determine all $2$-local extensions of $\eta_1$ to $VT_n$ (Theorem~\ref{n3}) while showing that no $2$-local extensions exist to $WT_n$ (Theorem~\ref{thmmm}). In Section~4, we construct a family of matrix representations of $T_n$ (Proposition~\ref{porp}) and prove their reducibility (Theorem~\ref{rrrrr}), and we show that these representations are faithful for $n=2$ and unfaithful for all $n\geq 3$ (Theorem~\ref{fa2}). Finally, we identify all $2$-local extensions of these families to $VT_n$ and $WT_n$ for all $n\geq 3$ (Theorems~\ref{n33} and \ref{n333}).
\vspace{0.1cm}

\section{Generalities} 

For $n \geq 2$, the twin group, namely $T_n$, has generators $s_1,s_2,\ldots,s_{n-1}$ that satisfy the following relations.
\begin{equation} \label{eqs1}
\hspace{1.46cm} s_i^2=1,\hspace{0.95cm} i=1,2,\ldots,n-2,
\end{equation}
\begin{equation} \label{eqs2}
s_is_j=s_js_i, \hspace{0.55cm} |i-j|\geq 2.
\end{equation}
In particular, we see that $T_2=\langle s_1 \ |\ s_1^2=1 \rangle =\mathbb{Z}_2$ is the cyclic group of order $2$ and $T_3=\langle s_1,s_2 \ |\ s_1^2=s_2^2=1 \rangle=\mathbb{Z}_2*\mathbb{Z}_2$ is the infinite dihedral group.

\vspace*{0.1cm}

On the other hand, for $n\geq 2$, the virtual twin group, denoted by $VT_n$, is a group extension of $T_n$ that is generated by the generators $s_1,s_2, \ldots,s_{n-1}$ of $T_n$ and the generators $\rho_1,\rho_2, \ldots, \rho_{n-1}$. In addition to the relations (\ref{eqs1}) and (\ref{eqs2}) of $T_n$, the generators $s_i$ and $\rho_i$, $1\leq i \leq n-1$, of $VT_n$ satisfy the following relations.
\begin{equation} \label{eqs3}
\ \ \ \ \rho_i\rho_{i+1}\rho_i = \rho_{i+1}\rho_i\rho_{i+1}, \hspace{0.55cm} i=1,2,\ldots,n-2,
\end{equation}
\begin{equation} \label{eqs4}
\rho_i\rho_j = \rho_j\rho_i ,\hspace{1.5cm} |i-j|\geq 2,
\end{equation}
\begin{equation} \label{eqs5}
\ \ \ \ \ \ \ \ \ \ \ \ \rho_i^2 = 1 ,\hspace{2cm} i=1,2,\ldots,n-1,
\end{equation}
\begin{equation} \label{eqs6}
s_i\rho_j=\rho_js_i ,\hspace{1.55cm} |i-j|\geq 2,
\end{equation}
\begin{equation} \label{eqs7}
\ \ \ \ \rho_i\rho_{i+1}s_i=s_{i+1}\rho_i\rho_{i+1}, \hspace{0.6cm} i=1,2,\ldots,n-2. \vspace*{0.1cm}
\end{equation}
The quotient of $VT_n$ by adding the relations 
\begin{equation} \label{eqs8}
\ \ \ \ \rho_is_{i+1}s_i=s_{i+1}s_i\rho_{i+1}, \hspace{0.6cm} i=1,2,\ldots,n-2,
\end{equation}
is called the welded twin group and denoted by $WT_n$, which means that the generators of $WT_n$ satisfy all the relations of $VT_n$ and the relations (\ref{eqs8}) in addition. Notice that the relations (\ref{eqs6}), (\ref{eqs7}), and (\ref{eqs8}) are called the mixed relations. In addition, for $i=1,2,\ldots,n-2$, the relations (\ref{eqs7}) are equivalent to the relations
$$\rho_{i+1}\rho_is_{i+1}=s_i\rho_{i+1}\rho_i,$$
and the relations (\ref{eqs8}) are equivalent to the relations 
$$s_{i+1}\rho_i\rho_{i+1}=\rho_i\rho_{i+1}s_i.$$

\vspace*{0.1cm}

In what follows, we give the main definition of $k$-local representations of a group $G$ with finite number of generators.

\begin{definition} \cite{Nasserlocal}
Let $G$ be a group with generators $a_1,a_2,\ldots,a_{n-1}$. A representation $\theta: G \rightarrow \mathrm{GL}_{m}(\mathbb{Z}[t^{\pm 1}])$ is said to be $k$-local if it is of the form
$$\theta(a_i) =\left( \begin{array}{c|@{}c|c@{}}
   \begin{matrix}
     I_{i-1} 
   \end{matrix} 
      & 0 & 0 \\
      \hline
    0 &\hspace{0.2cm} \begin{matrix}
   		M_i
   		\end{matrix}  & 0  \\
\hline
0 & 0 & I_{n-i-1}
\end{array} \right) \hspace*{0.2cm} \text{for} \hspace*{0.2cm} 1\leq i\leq n-1,$$ 
where $M_i \in \mathrm{GL}_k(\mathbb{Z}[t^{\pm 1}])$ with $k=m-n+2$ and $I_k$ is the $k\times k$ identity matrix. The representation $\theta$ is said to be homogeneous if all the matrices $M_i$ are equal.
\end{definition}

Remark that if $G'$ is a group with $2(n-1)$ generators $a_1,a_2,\ldots,a_{n-1}$ and $b_1,b_2,\ldots,b_{n-1}$, then, the concept of $k$-local representations could be extended in the following way. A $k$-local representation $\theta: G' \rightarrow \mathrm{GL}_{m}(\mathbb{Z}[t^{\pm 1}])$ is a representation of the form
$$\theta(a_i) =\left( \begin{array}{c|@{}c|c@{}}
   \begin{matrix}
     I_{i-1} 
   \end{matrix} 
      & 0 & 0 \\
      \hline
    0 &\hspace{0.2cm} \begin{matrix}
   		M_i
   		\end{matrix}  & 0  \\
\hline
0 & 0 & I_{n-i-1}
\end{array} \right) \text{ and  } \ \theta(b_i) =\left( \begin{array}{c|@{}c|c@{}}
   \begin{matrix}
     I_{i-1} 
   \end{matrix} 
      & 0 & 0 \\
      \hline
    0 &\hspace{0.2cm} \begin{matrix}
   		N_i
   		\end{matrix}  & 0  \\
\hline
0 & 0 & I_{n-i-1}
\end{array} \right) $$
for $1\leq i\leq n-1,$ where $M_i,N_i \in \mathrm{GL}_k(\mathbb{Z}[t^{\pm 1}])$ with $k=m-n+2$ and $I_k$ is the $k\times k$ identity matrix. In this case, $\theta$ is homogeneous if all the matrices $M_i$ are equal and all the matrices $N_i$ are equal.

\vspace*{0.2cm}

We now introduce the Fox derivatives in the following definition.

\begin{definition}
Let $\mathbb{F}_n$ be a free group of rank $n$ with free basis $x_1,x_2,\ldots,x_n$. We define, for $k=1,2,\ldots,n,$ the Fox derivatives on the group ring $\mathbb{Z}[\mathbb{F}_n]$ as follows.
\begin{itemize}
\item[(a)] $\dfrac{\partial x_i}{\partial x_k}=\delta_{ik}$,
\item[(b)] $\dfrac{\partial x^{-1}_i}{\partial x_k}=-\delta_{ik}x^{-1}_i$,
\item[(c)] $\dfrac{\partial}{\partial x_k}(ab)= \dfrac{\partial a}{\partial x_k}\epsilon(a)+a\dfrac{\partial b}{\partial x_k}$, $a,b \in \mathbb{Z}[\mathbb{F}_n]$.
\end{itemize}
Note that $\epsilon(a)=1$ if $a\in \mathbb{F}_n$. Here $\delta_{ik}$ is the Kronecker symbol.
\end{definition}

\vspace{0.1cm}

\section{The representation $\eta_1$ of $T_n$} 

Let $\mathbb{F}_n$ be a free group of rank $n\in \mathbb{N}^*$ with $n$ generators $x_1,x_2,\ldots, x_n$ and let $\mathrm{Aut}(\mathbb{F}_n)$ be the automorphism group of $\mathbb{F}_n$. In this section, we construct a representation of $T_n$ to $\mathrm{Aut}(\mathbb{F}_n)$ and we study some of its characteristics. Moreover, we study the existence of extensions of this representation to $VT_n$ and $WT_n$. 

\vspace{0.1cm}

\subsection{Constructing the representation} 

Consider the mapping $\eta_1: T_n \rightarrow \mathrm{Aut}(\mathbb{F}_n)$ given as follows.
$$\eta_1(s_i):
\left\{\begin{array}{l}
x_i\rightarrow x_ix_{i+1}x_i^{-1},\\
x_{i+1}\rightarrow x_ix_{i+1}^{-1}x_ix_{i+1}x_i^{-1},\\
x_k\rightarrow x_k \text{ for } k \notin \{i, i+1\}. \\
\end{array}\right.$$

\begin{proposition} \label{prop1}
The mapping $\eta_1$ defines a representation of $T_n$.
\end{proposition}

\begin{proof}
To prove that $\eta_1$ defines a representation of $T_n$, it is sufficient to prove that $\eta_1$ preserves the two defining relations of $T_n$. We consider these two relations of in the following cases.
\begin{itemize}
\item[(a)] \underline{$\eta_1(s_i^2)(x_k)=x_k$ for $1\leq i \leq n-1$ and $1\leq k \leq n-1$}: To prove this, we have three sub cases.\vspace*{0.15cm}
\begin{itemize}
\item[•] $\eta_1(s_i^2)(x_i)=\eta_1(s_i)(x_ix_{i+1}x_i^{-1}) \\ \hspace*{1.52cm}=(x_ix_{i+1}x_i^{-1})(x_ix_{i+1}^{-1}x_ix_{i+1}x_i^{-1})(x_ix_{i+1}^{-1}x_i^{-1})\\ \hspace*{1.52cm}=x_i.$\vspace*{0.15cm}
\item[•] $\eta_1(s_i^2)(x_{i+1})=\eta_1(s_i)(x_ix_{i+1}^{-1}x_ix_{i+1}x_i^{-1}) \\ \hspace*{1.875cm}= (x_ix_{i+1}x_i^{-1})(x_ix_{i+1}^{-1}x_i^{-1}x_{i+1}x_i^{-1})(x_ix_{i+1}x_i^{-1})\\ \hspace*{2.33cm}(x_ix_{i+1}^{-1}x_ix_{i+1}x_i^{-1})(x_ix_{i+1}^{-1}x_i^{-1}) \\ \hspace*{1.875cm}= x_{i+1}.$\vspace*{0.15cm}
\item[•] $\eta_1(s_i^2)(x_{k})=\eta_1(s_i)(x_{k})=x_k$ for all $k \notin \{i, i+1\}$.\vspace*{0.15cm}
\end{itemize} 
\item[(b)] \underline{$\eta_1(s_i s_j)(x_k) = \eta_1(s_j s_i)(x_k)$ for $|i-j|\geq 2$ and $1 \leq k \leq n-1$}: This case is straightforward by the construction of $\eta_1$, since $s_i$ and $s_j$ act on disjoint sets of generators in $\mathbb{F}_n$ for $|i-j|\geq 2$, so their actions commute.
\end{itemize}
\vspace*{0.15cm}
Thus, $\eta_1$ preserves the relations of $T_n$, and so, $\eta_1$ defines a representation of $T_n$.
\end{proof}

In what follows, we construct a matrix representation of $T_n$ derived for $\eta_1$, by means of Magnus representation using Fox derivatives. Let $D_k=\dfrac{\partial}{\partial x_k}$, for $1 \leq k \leq n$. We determine the Jacobian matrix of the image of the generators $s_i$ as follows.

$$J(s_i) = \left( \begin{matrix}
D_1(s_i(x_1)) & \dots  & D_n(s_i(x_1))\\
\vdots &   & \vdots & \\
D_1(s_i(x_n)) & \dots & D_n(s_i(x_n)) \end{matrix} \right).$$

\begin{proposition} \label{proppp}
The matrix representation of $T_n$ derived for $\eta_1$ is the representation where the images of the generators $s_i$ of $T_n$ are as follows. For $1\leq i \leq n-1$,
$$s_i\rightarrow \left( \begin{array}{c|@{}c|c@{}}
   \begin{matrix}
     I_{i-1} 
   \end{matrix} 
      & 0 & 0 \\
      \hline
    0 &\hspace{0.2cm} \begin{matrix}
   	1-x_ix_{i+1}x_i^{-1} & x_i\\
   	1+x_ix_{i+1}^{-1}-x_ix_{i+1}^{-1}x_ix_{i+1}x_i^{-1} & -x_ix_{i+1}^{-1}+x_ix_{i+1}^{-1}x_i\\
\end{matrix}  & 0  \\
\hline
0 & 0 & I_{n-i-1}
\end{array} \right).$$ 
\end{proposition}

\begin{proof}
Fix $1\leq i \leq n-1$. We aim to compute the Jacobian matrix $J(s_i)$ given above. We have the following five cases.
\begin{itemize}
\item[(a)] $D_i(s_i(x_{i}))=D_i(x_ix_{i+1}x_i^{-1})\\
\hspace*{1.55cm}=D_i(x_i)+x_iD_i(x_{i+1}x_i^{-1})\\
\hspace*{1.55cm}=1+x_i(D_i(x_{i+1})+x_{i+1}D_i(x_i^{-1}))\\
\hspace*{1.55cm}=1+x_i(-x_{i+1}x_i^{-1})\\
\hspace*{1.55cm}=1-x_ix_{i+1}x_i^{-1}.$\vspace*{0.1cm}
\item[(b)] $D_i(s_i(x_{i+1}))=D_i(x_ix_{i+1}^{-1}x_ix_{i+1}x_i^{-1})\\
\hspace*{1.9cm}=D_i(x_i)+x_iD_i(x_{i+1}^{-1}x_ix_{i+1}x_i^{-1})\\
\hspace*{1.9cm}=1+x_i(D_i(x_{i+1}^{-1})+x_{i+1}^{-1}D_i(x_ix_{i+1}x_i^{-1}))\\
\hspace*{1.9cm}=1+x_i(x_{i+1}^{-1}(D_i(x_i)+x_iD_i(x_{i+1}x_i^{-1})))\\
\hspace*{1.9cm}=1+x_i(x_{i+1}^{-1}(1+x_i(D_i(x_{i+1})+x_{i+1}D_i(x_i^{-1}))))\\
\hspace*{1.9cm}=1+x_i(x_{i+1}^{-1}(1+x_i(-x_{i+1}x_i^{-1})))\\
\hspace*{1.9cm}=1+x_i(x_{i+1}^{-1}(1-x_ix_{i+1}x_i^{-1}))\\
\hspace*{1.9cm}=1+x_ix_{i+1}^{-1}-x_ix_{i+1}^{-1}x_ix_{i+1}x_i^{-1}.$\vspace*{0.1cm}
\item[(c)] $D_{i+1}(s_i(x_{i}))=D_{i+1}(x_ix_{i+1}x_i^{-1})\\
\hspace*{1.9cm}=D_{i+1}(x_i)+x_iD_{i+1}(x_{i+1}x_i^{-1})\\
\hspace*{1.9cm}=x_i(D_{i+1}(x_{i+1})+x_{i+1}D_{i+1}(x_i^{-1}))\\
\hspace*{1.9cm}=x_i.$\vspace*{0.1cm}
\item[(d)] $D_{i+1}(s_i(x_{i+1}))=D_{i+1}(x_ix_{i+1}^{-1}x_ix_{i+1}x_i^{-1})\\
\hspace*{2.27cm}=D_{i+1}(x_i)+x_iD_{i+1}(x_{i+1}^{-1}x_ix_{i+1}x_i^{-1})\\
\hspace*{2.27cm}=x_i(D_{i+1}(x_{i+1}^{-1})+x_{i+1}^{-1}D_{i+1}(x_ix_{i+1}x_i^{-1}))\\
\hspace*{2.27cm}=x_i(-x_{i+1}^{-1}+x_{i+1}^{-1}(D_{i+1}(x_i)+x_iD_{i+1}(x_{i+1}x_i^{-1})))\\
\hspace*{2.27cm}=x_i(-x_{i+1}^{-1}+x_{i+1}^{-1}(x_i(D_{i+1}(x_{i+1})+x_{i+1}D_{i+1}(x_i^{-1}))))\\
\hspace*{2.27cm}=x_i(-x_{i+1}^{-1}+x_{i+1}^{-1}(x_i))\\
\hspace*{2.27cm}=x_i(-x_{i+1}^{-1}+x_{i+1}^{-1}x_i)\\
\hspace*{2.27cm}=-x_ix_{i+1}^{-1}+x_ix_{i+1}^{-1}x_i.$\vspace*{0.1cm}
\item[(e)] $D_{k}(s_i(x_{l}))=\delta_{kl}$ for any $1\leq k,l\leq n-1$ with $k,l \notin \{i,i+1\}$.\vspace*{0.1cm}
\end{itemize}
Substituting the obtained results in $J(s_i)$ implies our required result.
\end{proof}

In what follows, we consider a homomorphism $\tau:\mathbb{F}_n \rightarrow \mathbb{Z}[t^{\pm 1}]$, where $t$ is indeterminate, and we let $\tau(x_1)=\tau(x_2)=\cdots=\tau(x_n)=t$ to obtain the following matrix representation of $T_n$ to $\mathrm{GL}_n(\mathbb{Z}[t^{\pm 1}])$. We still call it $\eta_1$ because it is derived from the representation $\eta_1: T_n \rightarrow \mathrm{Aut}(\mathbb{F}_n)$ defined above.

\begin{proposition} \label{parop}
There is a matrix representation $\eta_1: T_n \rightarrow \mathrm{GL}_n(\mathbb{Z}[t^{\pm 1}])$ given by acting on the generators of $T_n$ as follows.
$$s_i\rightarrow \left( \begin{array}{c|@{}c|c@{}}
   \begin{matrix}
     I_{i-1} 
   \end{matrix} 
      & 0 & 0 \\
      \hline
    0 &\hspace{0.2cm} \begin{matrix}
   	1-t & t\\
   	2-t & t-1\\
\end{matrix}  & 0  \\
\hline
0 & 0 & I_{n-i-1}
\end{array} \right) \mathrm{ for } 1\leq i \leq n-1.$$ 
\end{proposition}

Form the shape of the representation $\eta_1: T_n \rightarrow \mathrm{GL}_n(\mathbb{Z}[t^{\pm 1}])$, we can see that it is a homogeneous $2$-local representation. In order to generalize our result, we end this subsection with the following question.

\begin{question}
Consider $n\geq 2$ and let $\eta: T_n \rightarrow \mathrm{GL}_n(\mathbb{Z}[t^{\pm 1}])$ be a homogeneous $2$-local representation of $T_n$. What are the possible forms of $\eta$?
\end{question}

\subsection{The irreducibility of $\eta_1$} 

In this part of the paper, we study the irreducibility of the representation $\eta_1: T_n \rightarrow \mathrm{GL}_n(\mathbb{Z}[t^{\pm 1}])$. We start by the following theorem.

\begin{theorem} \label{reddd}
The representation $\eta_1: T_n \rightarrow \mathrm{GL}_n(\mathbb{Z}[t^{\pm 1}])$ is reducible.
\end{theorem}

\begin{proof}
We can clearly see that the vector $(1,1,\ldots,1)^T$ is invariant under $\eta_1(s_i)$ for all $1\leq i \leq n-1$. Therefore, $\eta_1$ is reducible.
\end{proof}

Now, we specialize $t$ to a nonzero complex number and we find a composition factor of degree $n-1$ of $\eta_1$, namely $\eta_1'$. For that, we consider the following new basis for $\mathbb{C}^{n}$, $W=\{v_1,v_2,\ldots, v_n\}$, where $v_i=e_i$ for all $1 \leq i \leq n-1$ and $v_n=(1,1,\ldots,1)^T=e_1+e_2+\ldots+e_{n}$. Here $\{e_1,e_2,\ldots, e_{n}\}$ is the canonical basis of $\mathbb{C}^{n}$. We compute $\eta_1$ in the new basis $W$ as follows.\vspace{0.1cm}

\noindent \underline{For $1\leq i \leq n-2$}:\vspace{0.1cm} \\
$\eta_1(s_i)(v_1)=v_1$,\\
\hspace*{0.7cm} \vdots\\
$\eta_1(s_i)(v_{i-1})=v_{i-1}$,\\
$\eta_1(s_i)(v_{i})=(1-t)v_i+tv_{i+1}$,\\
$\eta_1(s_i)(v_{i+1})=(2-t)v_i+(t-1)v_{i+1}$,\\
$\eta_1(s_i)(v_{i+2})=v_{i+2}$,\\
\hspace*{0.7cm} \vdots\\
$\eta_1(s_i)(v_{n})=v_{n}$.\vspace*{0.1cm}

\noindent \underline{For $i=n-1$}:\vspace{0.1cm} \\
$\eta_1(s_{n-1})(v_1)=v_1$,\\
\hspace*{0.7cm} \vdots\\
$\eta_1(s_{n-1})(v_{n-2})=v_{n-2}$,\\
$\eta_1(s_{n-1})(v_{n-1})=(1-t)v_{n-1}+(2-t)e_n\\
\hspace*{2.31cm}=(1-t)v_{n-1}+(2-t)(-v_1-v_2-\ldots-v_{n-1}+v_n)\\
\hspace*{2.31cm}=(t-2)v_1+(t-2)v_2+\ldots+(t-2)v_{n-2}-v_{n-1}+(2-t)v_n,$\\
$\eta_1(s_{n-1})(v_{n})=v_{n}$.\vspace{0.1cm}\\
Thus, the representation $\eta_1: T_n \rightarrow \mathrm{GL}_n(\mathbb{Z}[t^{\pm 1}])$ in the new basis is given by acting on the generators of $T_n$ as follows.
$$s_i\rightarrow \left( \begin{array}{c|@{}c|c@{}}
   \begin{matrix}
     I_{i-1} 
   \end{matrix} 
      & 0 & 0 \\
      \hline
    0 &\hspace{0.2cm} \begin{matrix}
   	1-t & t\\
   	2-t & t-1\\
\end{matrix}  & 0  \\
\hline
0 & 0 & I_{n-i-1}
\end{array} \right) \text{ for } 1\leq i \leq n-2,$$
and  
$$s_{n-1}\rightarrow \left( \begin{array}{c|@{}c c@{}}
   \begin{matrix}
     I_{n-2} 
   \end{matrix} 
      & 0  \\
      \hline
    \begin{matrix}
   		t-2 &  t-2 & \ldots & t-2  \\
   		0 &  0 & \ldots & 0 \\
   		\end{matrix}  \ &\hspace{0.2cm} \begin{matrix}
   		-1 &  2-t  \\
   		0 &  1  \\   		
   		\end{matrix}  \\
\end{array} \right).\vspace*{0.2cm}$$
Eliminating the last row and the last column in each matrix above, we obtain the composition factor $\eta_1':T_n \rightarrow \mathrm{GL}_{n-1}(\mathbb{Z}[t^{\pm 1}])$ of $\eta_1$ given by acting on the generators of $T_n$ as follows.
$$s_i\rightarrow \left( \begin{array}{c|@{}c|c@{}}
   \begin{matrix}
     I_{i-1} 
   \end{matrix} 
      & 0 & 0 \\
      \hline
    0 &\hspace{0.2cm} \begin{matrix}
   	1-t & t\\
   	2-t & t-1\\
\end{matrix}  & 0  \\
\hline
0 & 0 & I_{n-i-2}
\end{array} \right) \text{ for } 1\leq i \leq n-2,$$
and  
$$s_{n-1}\rightarrow \left( \begin{array}{c|@{}c c@{}}
   \begin{matrix}
     I_{n-2} 
   \end{matrix} 
      & 0  \\
      \hline
    \begin{matrix}
   		t-2 &  t-2 & \ldots & t-2  \\
   		\end{matrix}  \ &\hspace{0.2cm} \begin{matrix}
   		-1  \\
   		\end{matrix}  \\
\end{array} \right).\vspace*{0.1cm}$$

Recall that we specialize $t$ to a nonzero complex number. We now determine the necessary and sufficient conditions for the irreducibility of the corresponding complex specialization of $\eta_1'$. Note that in the case $t = 0$, all matrices of the representation $\eta_1$ are lower triangular and therefore trivially reducible. For this reason, we exclude this case from our analysis.

\begin{theorem} \label{reddd1}
The representation $\eta_1': T_n \rightarrow \mathrm{GL}_{n-1}(\mathbb{C})$, where $t\neq 0$, is irreducible if and only if $t\neq \dfrac{2n-2}{n-2}$ and $t\neq 2$.
\end{theorem}

\begin{proof}
For the necessary condition, suppose that $t=\dfrac{2n-2}{n-2}$ or $t=2$. Require to prove that $\eta_1'$ is reducible. For that, we consider each case separately.
\begin{itemize}
\item[(a)] In the case $t=\dfrac{2n-2}{n-2}$, we can see that the vector $(1,1,\ldots,1)^T$ is invariant under $\eta'_1(s_i)$ for all $1\leq i \leq n-1$, and so $\eta_1$ is reducible, as required.
\item[(b)] In the case $t=2$, we can see that the vector $(0,\ldots,0,1)^T$ is invariant under $\eta'_1(s_i)$ for all $1\leq i \leq n-1$. Therefore, $\eta_1$ is reducible, as required.
\end{itemize}

Now, for the sufficient condition, suppose that $t\neq \dfrac{2n-2}{n-2}$ and $t\neq 2$. Assume to get a contradiction that $\eta_1'$ is reducible and let $U$ be a nontrivial subspace of $\mathbb{C}^{n-1}$ that is invariant under $\eta_1'$. Pick $u=(u_1,u_2,\ldots, u_{n-1})\in U$ in a way that there exists $1\leq j \leq n-2$ such that $u_j\neq u_{j+1}$. Such $u$ clearly exits, since otherwise we get that $U$ is generated by the vector $(1,1,\ldots,1)^T$, and then we get that $t=\dfrac{2n-2}{n-2}$ as $U$ is invariant under $\eta_1'(s_i)$ for all $1\leq i \leq n-1$, which is impossible. We now consider two cases as follows.
\begin{itemize}
\item \underline{The case $j=n-2$}. In this case, we have $u_{n-2}\neq u_{n-1}$. Now, as $U$ is invariant under $\eta_1'(s_i)$ for all $1\leq i \leq n-1$, we have
$$\eta_1'(s_{n-2})(u)-u=(-u_{n-2}+u_{n-1})\left(te_{n-2}+(t-2)e_{n-1}\right) \in U.$$
But $u_{n-2}\neq u_{n-1}$ implies that 
$$w=te_{n-2}+(t-2)e_{n-1} \in U.$$
Now, we have that 
$$\eta_1'(s_{n-1})(w)=te_{n-2}+(t-1)(t-2)e_{n-1} \in U$$
and so
$$\eta_1'(s_{n-1})(w)-w=(t-2)^2e_{n-1}\in U.$$
But $t\neq 2$, which implies that $$e_{n-1}\in U.$$
On the other hand, we can see that 
\begin{equation} \label{111}
\eta_1'(s_{j-1})(e_{j})-(t-1)e_{j}=te_{j-1}
\end{equation}
for all $2\leq j \leq n-1$. So, as $e_{n-1}\in U$ and $t\neq 0$, direct computations using $n-2$ iterations in Equation (\ref{111}) imply that $e_{j-1} \in U$ for all $2\leq j \leq n-1$. Thus, we get that $e_{j} \in U$ for all $1\leq j \leq n-1$, and so $U=\mathbb{C}^{n-1}$, a contradiction. Therefore, $\eta_1'$ is irreducible in this case.\vspace*{0.1cm}
\item \underline{Case $j < n-2$}. 
Following the same argument as above, we find
\[
(2t-2)(t-2) e_{n-1} \in U.
\]
Since $t \neq 2$ and, as shown in the previous case using Equation \eqref{111}, $e_{n-1} \notin U$, we get  $2t-2 = 0$, which gives $t = 1$. Then, as $U$ is invariant under the transposition matrices $\eta_1'(s_k)$ for $k<n-1$, we get that $e_j - e_{j+1} \in U$. Applying these matrices repeatedly, we obtain $e_{n-2} - e_{n-1} \in U$, and applying $\eta_1'(s_{n-1})$ to this vector yields $e_{n-1} \in U$, leading again to a contradiction as in the previous case. Therefore, $\eta_1'$ is irreducible in this case as well, and the proof is completed.
\end{itemize}
\end{proof}

\subsection{The faithfulness of $\eta_1$} 

In this subsection, we aim to present a result on the faithfulness of the representation $\eta_1$. More deeply, we prove the following theorem.

\begin{theorem} \label{fa1}
The representation $\eta_1: T_n \rightarrow \mathrm{GL}_n(\mathbb{Z}[t^{\pm 1}])$ is faithful for $n=2$ and $n=3$.
\end{theorem}

\begin{proof}
For $n=2$, it is trivial that the representation $\eta_2$ is faithful as $T_2=\{e,s_1\}$ and $\eta_2(s_1)\neq I_2$. Now, for $n=3$, we have $T_3=\langle s_1, s_2\rangle$ with $s_1^2=s_2^2=1$. So, the possible elements in $\ker(\eta_1)$ are either $s_1s_2s_1s_2\ldots$ or $s_2s_1s_2s_1\ldots$. But $\det(s_1)=\det(s_2)=-1$, so the possible elements in $\ker(\eta_1)$ should be of even length. Hence, the possible elements in $\ker(\eta_1)$ are either $(s_1s_2)^r$ or $(s_2s_1)^r$. Remark that $(s_1s_2)^r\in \ker(\eta_1)$ if and only if $(s_2s_1)^r \in \ker(\eta_1)$ since
both elements are conjugate. Now, assume to get a contradiction that $\eta_1$ is unfaithful in the case $n=3$, then there exists $r\in \mathbb{N}^*$ such that $(s_1s_2)^r\in \ker(\eta_1)$. Using Wolfram Mathematica, we get that 
$$\eta_1((s_1s_2)^r)=\left( \begin{matrix}
a & b & c\\
d & e & f\\
g & h & i
\end{matrix} \right),$$
where \vspace*{0.2cm}\\
$d-g=\frac{2^{-r} (t-2) \left(\left(-t^2+2t-2-\sqrt{(t-2) ((t-2) t+4)} \sqrt{t}\right)^r-\left(-t^2+2t-2+\sqrt{(t-2) ((t-2) t+4)} \sqrt{t}\right)^r\right)}{\sqrt{t} \sqrt{(t-2) ((t-2) t+4)}}.$
   But we have $d=g=0$ since $(s_2s_1)^r\in \ker(\eta_1)$, which implies that \vspace*{0.2cm}
 \\
   $\frac{2^{-r} (t-2) \left(\left(-t^2+2t-2-\sqrt{(t-2) ((t-2) t+4)} \sqrt{t}\right)^r-\left(-t^2+2t-2+\sqrt{(t-2) ((t-2) t+4)}\sqrt{t}\right)^r\right)}{\sqrt{t} \sqrt{(t-2) ((t-2) t+4)}}=0.$  \vspace*{0.2cm}
\\
   Notice that $t$ is indeterminate, so we get that \vspace*{0.2cm}\\
   $\left(-t^2+2 t-2-\sqrt{(t-2) ((t-2) t+4)} \sqrt{t}\right)^r=\left(-t^2+2t-2+\sqrt{(t-2) ((t-2) t+4)}
   \sqrt{t}\right)^r$. \vspace*{0.2cm}\\
   Thus, we get that $\sqrt{(t-2) ((t-2) t+4)} \sqrt{t}=0$, which is a contradiction as $t$ is indeterminate. Thus, $\eta_1$ is faithful in this case.
\end{proof}

We end this subsection with the following question.

\begin{question}
Is it true that the representation $\eta_1: T_n \rightarrow \mathrm{GL}_n(\mathbb{Z}[t^{\pm 1}])$ is faithful for $n\geq 4$?
\end{question}

\subsection{On extending $\eta_1$ to $VT_n$ and $WT_n$}

We aim to extend the representation $\eta_1: T_n \rightarrow \mathrm{GL}_n(\mathbb{Z}[t^{\pm 1}])$ to $VT_n$ and $WT_n$ in the $2$-local extension way, as these are two group extensions of $T_n$. However, the query is: do such extensions exist? We prove, in this subsection, that such extensions to $VT_n$ exist for all $n\geq 2$, but this is not the case for $WT_n$ in general.

First, in the following theorem, we find the shape of all extensions of $\eta_1$ to $VT_n$ and $WT_n$ when $n=2$. We have $VT_2=WT_2=\langle s_1,\rho_1\rangle$. Without loss of generality, we deal with $VT_2$ in the next theorem and it is the same for $WT_2$.

\begin{theorem} \label{n2}
Let $\hat{\eta}_1: VT_2 \rightarrow \mathrm{GL}_2(\mathbb{Z}[t^{\pm 1}])$ be a representation of $VT_2$ that extends $\eta_1$. Then, $\hat{\eta}_1$ is equivalent to one of the following five representations.
\begin{itemize}
\item[(1)] $\hat{\eta}_1^{(1)}: T_2 \rightarrow \mathrm{GL}_2(\mathbb{Z}[t^{\pm 1}])$ such that $$\hat{\eta}_1^{(1)}(s_1) =\left( \begin{array}{c@{}}
   \begin{matrix}
   		1-t & t\\
   		2-t & t-1
   		\end{matrix}
\end{array} \right)\hspace*{0.15cm} \text{and } \hspace*{0.15cm} \hat{\eta}_1^{(1)}(\rho_1) =\left( \begin{array}{c@{}}
   \begin{matrix}
   		a & b\\
   		\dfrac{1-a^2}{b} & -a
   		\end{matrix}
\end{array} \right),$$ where $a,b \in \mathbb{Z}[t^{\pm 1}], b \text{ is invertible in } \mathbb{Z}[t^{\pm 1}] \text{ and divides } 1-a^2.$\\
\item[(2)] $\hat{\eta}_1^{(2)}: T_2 \rightarrow \mathrm{GL}_2(\mathbb{Z}[t^{\pm 1}])$ such that $$\hat{\eta}_1^{(2)}(s_1) =\left( \begin{array}{c@{}}
   \begin{matrix}
   		1-t & t\\
   		2-t & t-1
   		\end{matrix}
\end{array} \right)\hspace*{0.15cm} \text{and } \hspace*{0.15cm} \hat{\eta}_1^{(2)}(\rho_1) =\left( \begin{array}{c@{}}
   \begin{matrix}
   		1 & 0\\
   		c & -1
   		\end{matrix}
\end{array} \right),$$ where $c \in \mathbb{Z}[t^{\pm 1}].$\\
\item[(3)] $\hat{\eta}_1^{(3)}: T_2 \rightarrow \mathrm{GL}_2(\mathbb{Z}[t^{\pm 1}])$ such that $$\hat{\eta}_1^{(3)}(s_1) =\left( \begin{array}{c@{}}
   \begin{matrix}
   		1-t & t\\
   		2-t & t-1
   		\end{matrix}
\end{array} \right)\hspace*{0.15cm} \text{and } \hspace*{0.15cm} \hat{\eta}_1^{(3)}(\rho_1) =\left( \begin{array}{c@{}}
   \begin{matrix}
   		-1 & 0\\
   		c & 1
   		\end{matrix}
\end{array} \right),$$ where $c \in \mathbb{Z}[t^{\pm 1}].$\\
\item[(4)] $\hat{\eta}_1^{(4)}: T_2 \rightarrow \mathrm{GL}_2(\mathbb{Z}[t^{\pm 1}])$ such that $$\hat{\eta}_1^{(4)}(s_1) =\left( \begin{array}{c@{}}
   \begin{matrix}
   		1-t & t\\
   		2-t & t-1
   		\end{matrix}
\end{array} \right)\hspace*{0.15cm} \text{and } \hspace*{0.15cm} \hat{\eta}_1^{(4)}(\rho_1) =\left( \begin{array}{c@{}}
   \begin{matrix}
   		-1 & 0\\
   		0 & -1
   		\end{matrix}
\end{array} \right).\vspace*{0.1cm}$$
\item[(5)] $\hat{\eta}_1^{(5)}: T_2 \rightarrow \mathrm{GL}_2(\mathbb{Z}[t^{\pm 1}])$ such that $$\hat{\eta}_1^{(5)}(s_1) =\left( \begin{array}{c@{}}
   \begin{matrix}
   		1-t & t\\
   		2-t & t-1
   		\end{matrix}
\end{array} \right)\hspace*{0.15cm} \text{and } \hspace*{0.15cm} \hat{\eta}_1^{(5)}(\rho_1) =\left( \begin{array}{c@{}}
   \begin{matrix}
   		1 & 0\\
   		0 & 1
   		\end{matrix}
\end{array} \right).$$
\end{itemize}
\end{theorem}

\begin{proof}
Set $\hat{\eta}_1(\rho_1)=\left( \begin{array}{c@{}}
   \begin{matrix}
   		a & b\\
   		c & d
   		\end{matrix}
\end{array} \right), \text{ where } a, b, c, d \in \mathbb{Z}[t^{\pm 1}].$ The only relation between the generators of $T_2$ is $\rho_1^2=1$, which implies that $\hat{\eta}_1(\rho_1)^2=I_2$. Applying this relation, we get the following system of equations.
\begin{equation} \label{eqq1}
a^2+bc=1,
\end{equation} 
\begin{equation} \label{eqq2}
bd+ab=0,
\end{equation} 
\begin{equation} \label{eqq3}
cd+ac=0,
\end{equation}
\begin{equation} \label{eqq4}
d^2+bc=1.
\end{equation} 
Now, subtract Equation (\ref{eqq1}) from Equation (\ref{eqq4}), we get that $a^2=d^2$, which gives that $a=\pm d$. We consider four cases in the following.
\begin{itemize}
\item $a=d\neq 0$. From Equation (\ref{eqq2}) and Equation (\ref{eqq3}), we get that $b=c=0$, and so, from Equation (\ref{eqq1}) and Equation (\ref{eqq4}), we get that $a^2=d^2=1$, which implies that $a=\pm1$ and $d=\pm1$. But we have $a=d$, and so we get that:
\begin{itemize}
\item $\hat{\eta}_1$ is equivalent to $\hat{\eta}_1^{(4)}$ if $a=d=-1$.
\item $\hat{\eta}_1$ is equivalent to $\hat{\eta}_1^{(5)}$ if $a=d=1$.
\end{itemize}
\item $a=d=0$. From Equation (\ref{eqq1}) and Equation (\ref{eqq4}), we get that $bc=1$, and so $\hat{\eta}_1$ is equivalent to a special case of $\hat{\eta}_1^{(1)}$.
\item $a=-d$ and $b\neq 0$. From Equation (\ref{eqq4}), we get that $bc=1-a^2$, and so $\hat{\eta}_1$ is equivalent to $\hat{\eta}_1^{(1)}$, as required.
\item $a=-d$ and $b=0$. From Equation (\ref{eqq1}) and Equation (\ref{eqq4}), we get that $a^2=1$ and $d^2=1$, and so $a=\pm1$ and $d=\mp 1$. Hence, we get that
\begin{itemize}
\item $\hat{\eta}_1$ is equivalent to a special case of $\hat{\eta}_1^{(2)}$ if $a=1$ and $d=-1$.
\item $\hat{\eta}_1$ is equivalent to a special case of $\hat{\eta}_1^{(3)}$ if $a=-1$ and $d=1$.
\end{itemize}
\end{itemize}
Therefore, $\hat{\eta}_1$ is equivalent to one of the five representations $\hat{\eta}_1^{(i)}$, $1\leq i \leq 5$, as required.
\end{proof}

We now find all $2$-local extensions of the representation $\eta_1$ of $T_n$ to $VT_n$ for all $n\geq 3$.

\begin{theorem} \label{n3}
Consider $n\geq 3$ and let $\hat{\eta}_1: VT_n \rightarrow \mathrm{GL}_n(\mathbb{Z}[t^{\pm 1}])$ be a $2$-local representation of $VT_n$ that extends $\eta_1$. Then, $\hat{\eta}_1$ has only one form, which is given by acting on the generators of $VT_n$ as follows.
$$\hat{\eta}_1(s_i)= \left( \begin{array}{c|@{}c|c@{}}
   \begin{matrix}
     I_{i-1} 
   \end{matrix} 
      & 0 & 0 \\
      \hline
    0 &\hspace{0.2cm} \begin{matrix}
   	1-t & t\\
   	2-t & t-1\\
\end{matrix}  & 0  \\
\hline
0 & 0 & I_{n-i-1}
\end{array} \right)$$
and 
$$\hat{\eta}_1(\rho_i)= \left( \begin{array}{c|@{}c|c@{}}
   \begin{matrix}
     I_{i-1} 
   \end{matrix} 
      & 0 & 0 \\
      \hline
    0 &\hspace{0.2cm} \begin{matrix}
   	0 & b\\
   	\dfrac{1}{b} & 0\\
\end{matrix}  & 0  \\
\hline
0 & 0 & I_{n-i-1}
\end{array} \right)$$
for $1 \leq i \leq n-1$, where $b \in \mathbb{Z}[t^{\pm 1}]$ is invertible.
\end{theorem}

\begin{proof}
The proof follows the same strategy as in Theorem \ref{n2}.  
The key relations used here are 
\[
\rho_1 \rho_2 \rho_1 = \rho_2 \rho_1 \rho_2, \quad \rho_1^2 = 1, \quad \text{and} \quad \rho_1 \rho_2 s_1 = s_2 \rho_1 \rho_2.
\]  
Applying the representation $\hat{\eta}_1$ to these relations yields a system of equations, and solving them in the same manner as in Theorem \ref{n2} implies our result.
\end{proof}

\begin{question}
Consider the representation $\hat{\eta}_1: VT_n \rightarrow \mathrm{GL}_n(\mathbb{Z}[t^{\pm 1}])$ given in Theorem \ref{n3}. Can we find necessary and sufficient conditions for the irreducibility of this representation as we did for $\eta_1$?
\end{question}

Now, we prove in the next theorem a different result regarding extending the representation $\eta_1$ to $WT_n$.

\begin{theorem} \label{thmmm}
There is no $2$-local extension for the representation $\eta_1$ of $T_n$ to $WT_n$ For all $n\geq 3$.
\end{theorem}

\begin{proof}
Consider $n\geq 3$ and suppose to get a contradiction that there exists a $2$-local representation $\hat{\eta}_1: WT_n \rightarrow \mathrm{GL}_n(\mathbb{Z}[t^{\pm 1}])$ that extends $\eta_1$. Considering the relations (\ref{eqs3}), \ldots, (\ref{eqs7}) of the generators of $WT_n$, we can see that $\hat{\eta}_1$ has only one form, which is given in Theorem \ref{n3}. That is, 
$$\hat{\eta}_1(s_i)= \left( \begin{array}{c|@{}c|c@{}}
   \begin{matrix}
     I_{i-1} 
   \end{matrix} 
      & 0 & 0 \\
      \hline
    0 &\hspace{0.2cm} \begin{matrix}
   	1-t & t\\
   	2-t & t-1\\
\end{matrix}  & 0  \\
\hline
0 & 0 & I_{n-i-1}
\end{array} \right)$$
and 
$$\hat{\eta}_1(\rho_i)= \left( \begin{array}{c|@{}c|c@{}}
   \begin{matrix}
     I_{i-1} 
   \end{matrix} 
      & 0 & 0 \\
      \hline
    0 &\hspace{0.2cm} \begin{matrix}
   	0 & b\\
   	\dfrac{1}{b} & 0\\
\end{matrix}  & 0  \\
\hline
0 & 0 & I_{n-i-1}
\end{array} \right)$$
for $1 \leq i \leq n-1$, where $b \in \mathbb{Z}[t^{\pm 1}]$ is invertible. Now, direct computations imply that $\hat{\eta_1}$ does not preserve the relations (\ref{eqs8}) of $WT_n$; that is:
$$\hat{\eta_1}(\rho_i)\hat{\eta_1}(s_{i+1})\hat{\eta_1}(s_i)\neq \hat{\eta_1}(s_{i+1})\hat{\eta_1}(s_i)\hat{\eta_1}(\rho_{i+1})$$
for $1\leq i \leq n-1$, which is a contradiction. Therefore, there is no extension for the representation $\eta_1$ of $T_n$ to $WT_n$ for all $n\geq 3$.
\end{proof}

\section{The representation $\eta_2$ of $T_n$} 

In this section, we aim to introduce a family of matrix representations of $T_n$, namely $\eta_2: T_n \rightarrow \mathrm{GL}_n(\mathbb{Z}[t^{\pm 1}])$, that can be extended to the group $WT_n$ in the $2$-local extension way. Let $f(t)\in \mathbb{Z}[t^{\pm 1}]$ be invertible. Consider the following mapping $\eta_2: T_n \rightarrow \mathrm{GL}_n(\mathbb{Z}[t^{\pm 1}])$ given by acting on the generators $s_i, 1\leq i \leq n-1,$ of $T_n$ as follows.
$$\eta_2(s_i)\rightarrow \left( \begin{array}{c|@{}c|c@{}}
   \begin{matrix}
     I_{i-1} 
   \end{matrix} 
      & 0 & 0 \\
      \hline
    0 &\hspace{0.2cm} \begin{matrix}
   	0 & f(t)\\
   	f(t)^{-1} & 0\\
\end{matrix}  & 0  \\
\hline
0 & 0 & I_{n-i-1}
\end{array} \right).$$

\begin{proposition} \label{porp}
The mapping $\eta_2$ defines a representation of $T_n$.
\end{proposition}  
\begin{proof}
We can directly see that $\eta_2$ preserves the defining relations of $T_n$, and therefore, $\eta_2$ defines a representation of $T_n$.
\end{proof}

We now study the irreducibility of the representation $\eta_2$.

\begin{theorem} \label{rrrrr}
The representation $\eta_2: T_n \rightarrow \mathrm{GL}_n(\mathbb{Z}[t^{\pm 1}])$ is reducible.
\end{theorem}

\begin{proof}
We construct an equivalent representation of $\eta_2$, namely $\hat{\eta_2}$. For this purpose, consider 
$$P = \mathrm{diag}\left(f(t)^{n-1}, f(t)^{n-2}, \ldots, f(t), 1\right),$$ where $\mathrm{diag}(r_1, r_2, \ldots, r_n)$ denotes the $n \times n$ diagonal matrix with diagonal entries $r_1, r_2, \ldots, r_n$. Define the representation $\hat{\eta_2}$ of $T_n$ by $\hat{\eta_2} = P^{-1}\eta_2P.$ 
Clearly, $\hat{\eta_2}$ is equivalent to $\eta_2$, and the action of $\hat{\eta_2}$ on the generators of $T_n$ is given as follows.
$$\hat{\eta_2}(s_i)= 
\left( \begin{array}{c|@{}c|c@{}}
   \begin{matrix} I_{i-1} \end{matrix} & 0 & 0 \\
   \hline
   0 & \hspace{0.2cm} \begin{matrix} 0 & 1\\ 1 & 0 \end{matrix} & 0 \\
   \hline
   0 & 0 & I_{n-i-1} 
\end{array} \right)$$
for $1 \leq i \leq n-1$. Direct computations imply that the column vector $(1,1,\ldots, 1)^T$ is invariant under $\hat{\eta_2}(s_i)$ for all $1\leq i \leq n-1$. Hence $\hat{\eta_2}$ is reducible, and so $\eta_2$ is reducible.
\end{proof}
 
Now, we present in the next theorem a result on the faithfulness of the representation $\eta_2$.

\begin{theorem} \label{fa2}
The representation $\eta_2: T_n \rightarrow \mathrm{GL}_n(\mathbb{Z}[t^{\pm 1}])$ is faithful for $n=2$ and unfaithful for all $n\geq 3$.
\end{theorem}

\begin{proof}
For $n=2$, it is trivial that the representation $\eta_2$ is faithful as $T_2=\{e,s_1\}$ and $\eta_2(s_1)\neq I_2$. Now, for $n\geq 3$, we can see that $\eta_2((s_is_{i+1})^3)=I_n$ with $(s_is_{i+1})^3$ are non-trivial elements of $T_n$, and so, $\eta_2$ is unfaithful.
\end{proof}

\subsection{On extending $\eta_2$ to $VT_n$ and $WT_n$}

Here, we aim to extend the representation $\eta_2: T_n \rightarrow \mathrm{GL}_n(\mathbb{Z}[t^{\pm 1}])$ to $VT_n$ and $WT_n$ in the $2$-local extension way. However, we prove that there is a family of $2$-local representations of $VT_n$ and $WT_n$ that extends $\eta_2$. Our work will be for $n\geq 3$ since for $n=2$ the result is similar to that for Theorem \ref{n2}.

\vspace*{0.15cm}

First, we find all $2$-local extensions of the representation $\eta_2$ to $VT_n$ for all $n\geq 3$.

\begin{theorem} \label{n33}
Consider $n\geq 3$ and let $\hat{\eta}_2: VT_n \rightarrow \mathrm{GL}_n(\mathbb{Z}[t^{\pm 1}])$ be a $2$-local representation of $VT_n$ that extends $\eta_2$. Then, $\hat{\eta}_2$ has only one form, which is given by acting on the generators of $VT_n$ as follows.
$$\hat{\eta}_2(s_i)= \left( \begin{array}{c|@{}c|c@{}}
   \begin{matrix}
     I_{i-1} 
   \end{matrix} 
      & 0 & 0 \\
      \hline
    0 &\hspace{0.2cm} \begin{matrix}
   	0 & f(t)\\
   	f(t)^{-1} & 0\\
\end{matrix}  & 0  \\
\hline
0 & 0 & I_{n-i-1}
\end{array} \right)$$
and 
$$\hat{\eta}_2(\rho_i)= \left( \begin{array}{c|@{}c|c@{}}
   \begin{matrix}
     I_{i-1} 
   \end{matrix} 
      & 0 & 0 \\
      \hline
    0 &\hspace{0.2cm} \begin{matrix}
 	0 & g(t)\\
   	g(t)^{-1} & 0\\
\end{matrix}  & 0  \\
\hline
0 & 0 & I_{n-i-1}
\end{array} \right)$$
for $1 \leq i \leq n-1$, where $g(t) \in \mathbb{Z}[t^{\pm 1}]$ is invertible.
\end{theorem}

\begin{proof}
The argument follows the same reasoning as in Theorems \ref{n2} and \ref{n3}, so we omit the details.
\end{proof}

Now, we find all $2$-local extensions of the representation $\eta_2$ to $WT_n$ for all $n\geq 3$.

\begin{theorem} \label{n333}
Consider $n\geq 3$ and let $\hat{\eta}_2: WT_n \rightarrow \mathrm{GL}_n(\mathbb{Z}[t^{\pm 1}])$ be a $2$-local representation of $WT_n$ that extends $\eta_2$. Then, $\hat{\eta}_2$ is equivalent to the representation given in Theorem \ref{n33}.
\end{theorem}
\begin{proof}
Considering the defining relations of $VT_n$, we see that the possible forms of $\hat{\eta}_2$ are those obtained in Theorem \ref{n33}.  Moreover, it is straightforward to verify that the relations \eqref{eqs8} of $WT_n$ are also satisfied for these representations, namely,
\[
\hat{\eta}_2(\rho_i)\, \hat{\eta}_2(s_{i+1})\, \hat{\eta}_2(s_i) = \hat{\eta}_2(s_{i+1})\, \hat{\eta}_2(s_i)\, \hat{\eta}_2(\rho_{i+1}).
\]  
Therefore, $\hat{\eta}_2$ is equivalent to the representation given in Theorem \ref{n33}, as claimed.
\end{proof}

\begin{question}
Consider the representation $\hat{\eta}_2: VT_n \rightarrow \mathrm{GL}_n(\mathbb{Z}[t^{\pm 1}])$ given in Theorem \ref{n33}. Under what conditions of $f(t)$ and $g(t)$ can we find necessary and sufficient conditions for the irreducibility of this representation?
\end{question}

\section{Conclusions}
We constructed two representations of the twin group $T_n$, $n\geq 2$, in two different approaches, which is our first major finding in this paper. The second major result is that we completely answered the question regarding the irreducibility of these representations. The third significant outcome is that we studied the faithfulness of these two representations in many cases. Finally, the fourth main result is that we investigated the possibility of extending these two representations to $VT_n$ and $WT_n$ in the $2$-local extension way. The paper addresses several questions for further research.

\end{document}